\documentclass{amsart}
\usepackage{amsfonts, color}
\usepackage{latexsym}
\usepackage{amssymb}
\usepackage{amsmath}

\newcommand{\R}{\mathbb R}
\newcommand{\N}{\mathbb N}

\newcommand{\E}{\mathbb E}
\newcommand{\Pro}{\mathbb P}

\newtheorem{thm}{Theorem}[section]
\newtheorem{cor}{Corollary}[section]
\newtheorem{lemma}{Lemma}[section]

\newtheorem{proposition}{Proposition}[section]

\theoremstyle{remark}
\newtheorem*{rmk}{Remark}
\DeclareMathOperator{\signum}{sgn}

\begin{document}


\title{The square negative correlation property on $\ell_p^n$- balls}

\author[D.\,Alonso]{David Alonso-Guti\'errez}
\address{\'Area de an\'alisis matem\'atico, Departamento de matem\'aticas, Facultad de Ciencias, Universidad de Zaragoza, Pedro Cerbuna 12, 50009 Zaragoza (Spain), IUMA}
\email[(David Alonso)]{alonsod@unizar.es}


\author[J.\,Bernu\'es]{Julio Bernu\'es}
\address{\'Area de an\'alisis matem\'atico, Departamento de matem\'aticas, Facultad de Ciencias, Universidad de Zaragoza, Pedro cerbuna 12, 50009 Zaragoza (Spain), IUMA}
\email[(Julio Bernu\'es)]{bernues@unizar.es}
\thanks{Partially supported by Spanish grants MTM2016-77710-P and DGA E-64}
\begin{abstract}
In this paper we prove that for any $p\in[2,\infty)$ the $\ell_p^n$ unit ball, $B_p^n$, satisfies the square negative correlation property with respect to every orthonormal basis, while we show it is not always the case for $1\le p\le 2$. In order to do that we regard $B_p^n$ as the orthogonal projection of $B_p^{n+1}$ onto the hyperplane $e_{n+1}^\perp$. We will also study the orthogonal projection of $B_p^n$ onto the hyperplane orthogonal to the diagonal vector $(1,\dots,1)$. In this case, the property holds for all $p\ge 1$ and $n$ large enough.
\end{abstract}

\date{\today}
\maketitle
\section{Introduction and notation}

A random vector $X$ on $\R^n$ is said to satisfy the square negative correlation property (SNCP) with respect to the orthonormal basis $\{\eta_i\}_{i=1}^n$ if for every $i\neq j$
$$
\E\langle X,\eta_i\rangle^2\langle X,\eta_j\rangle^2-\E\langle X,\eta_i\rangle^2\E\langle X,\eta_j\rangle^2\leq 0,
$$
where $\E$ denotes the expectation and $\langle\cdot,\cdot\rangle$ the standard scalar product on $\R^n$.

The study of the SNCP of random vectors uniformly distributed on convex bodies with respect to \textit{some} orthonormal basis appeared in \cite{ABP} in the context of the central limit problem for convex bodies, where the authors showed that for any $p\geq 1$ a random vector uniformly distributed on $B_p^n$ satisfies the SNCP with respect to the canonical basis $\{e_i\}_{i=1}^n$. In \cite{W}, this result was extended to random vectors uniformly distributed on generalized Orlicz balls, also with respect to the canonical basis. A straightforward consequence is that, by the rotational invariance of $B_2^n$, a random vector uniformly distributed on $B_2^n$  satisfies the SNCP with respect to \textit{every} orthonormal basis. The first non-trivial example in this new situation appeared in \cite{AB1}, where it was proved that any random vector uniformly distributed on any hyperplane projection of $B_\infty^n$ satisfies the SNCP with respect to \textit{every} orthonormal basis. In particular,  the SNCP with respect to \textit{every} orthonormal basis is satisfied by $B_\infty^n$ itself.  On the other hand, it is not hard to show that a random vector uniformly distributed on $B_1^n$ does not satisfy the SNCP with respect to every orthonormal basis (see Lemma \ref{lem:ValueOfF} below).

The relation between the SNCP and the central limit problem comes from the fact (see, for instance, \cite[Proposition 1.8] {AB2}) that if a zero-mean random vector uniformly distributed on a convex body $K$ in $\R^n$ satisfies the SNCP with respect to \textit{some} orthonormal basis, then it verifies the so called \textit{General Variance Conjecture} which states:

\textit{There exists an absolute constant $C$ such that for every zero-mean random vector $X$ uniformly distributed on a convex body
$$
\textrm{Var}|X|^2\leq C\lambda_X^2\E|X|^2,
$$
where $\lambda_X^2=\max_{\xi\in S^{n-1}}\E\langle X,\xi\rangle^2$ is the largest eigenvalue of the covariance matrix of $X$ and $Var$ denotes the variance. }

Here $|\cdot|$ denotes the Euclidean norm and $S^{n-1}$ denotes the unit Euclidean sphere $S^{n-1}=\{x\in\R^n\,:\,|x|=1\}$.

Furthermore, \cite[Proposition 1.9] {AB2}, if a zero-mean random vector uniformly distributed on $K$ satisfies the SNCP with respect to \textit{every} orthonormal basis, then $TK$ verifies the General Variance Conjecture for every linear isomorphism $T$ in $\R^n$.

This is a particular case of a well-known conjecture due to Kannan-Lov\'asz-Simonovits (see \cite{AB2}, for detailed explanations on this topic).

 \bigskip

In Section 3 we  study the SNCP on random vectors uniformly distributed on $B_p^n, p\ge 1$, with respect to any orthonormal basis. The main result is the following

\begin{thm}\label{thm:Bpn}
Let $X$ be a random vector uniformly distributed on $B_p^n$, $p\geq 1$, and write $\xi_1=\frac{e_1+e_2}{\sqrt{2}}$, $\xi_2=\frac{e_1-e_2}{\sqrt{2}}$. Let $f: S^{n-1}\times S^{n-1}$ be the function
$$
f(\eta_1,\eta_2)=\E\langle X,\eta_1\rangle^2\langle X,\eta_2\rangle^2-\E\langle X,\eta_1\rangle^2\E\langle X,\eta_2\rangle^2.
$$

Then for every $\eta_1,\eta_2\in S^{n-1}$ such that $\langle\eta_1,\eta_2\rangle=0$ we have,

$$
f(\xi_1,\xi_2)\leq f(\eta_1,\eta_2)\leq f(e_1,e_2),\qquad if \,\,\,\, p\geq 2
$$
$$
f(e_1,e_2)\leq f(\eta_1,\eta_2)\leq f(\xi_1,\xi_2),  \qquad if \,\,\,\,  p\leq 2.
$$

\end{thm}

Clearly, the choice of $e_1,e_2$ is not relevant as $f(e_1,e_2)=f(e_i,e_j), \forall\, i\ne j$. The analogous observation applies also to $\xi_1,\xi_2$.

We will compute $f(e_1,e_2)$ and $f(\xi_1,\xi_2)$ in Lemma 3.2 and express them in terms of the $\Gamma$ function in order to obtain the following
 \begin{cor}
 Let $p\geq1$ and $X$ a random vector uniformly distributed on $B_p^n$. \medskip
 \begin{itemize}
 	\item If $p\geq 2$, $X$ satisfies the SNCP with respect to every orthonormal basis.\medskip
 	\item If $1\leq p<2$, there exists $n_0(p)\in\N$ such that for any $n\geq n_0$ there is an orthonormal basis $\{\eta_i\}_{i=1}^n$ such that $X$ does not satisfy the SNCP with respect to $\{\eta_i\}_{i=1}^n$.
 \end{itemize}
 \end{cor}

 Moreover, we will show that $f(e_1,e_2)<0$ for all $p\ge 1$, providing a new proof of the aforementioned result in \cite{ABP}.

In order to prove Theorem \ref{thm:Bpn} we will view $B_p^n$ as the projection of $B_p^{n+1}$ onto the coordinate hyperplane $e_{n+1}^\perp$ orthogonal to $e_{n+1}$ and we will make use of the techniques developed in \cite{BN} and \cite{AB3}. The details of this approach are explained in Section 2.

In Section 4 we apply the same strategy to a random vector uniformly distributed on $P_{\theta_0^\perp}B_{p}^n$, the orthogonal projection of $B_p^n$ onto $\theta_0^\perp$, where $\theta_0=\left(\frac{1}{\sqrt n},\dots,\frac{1,  }{\sqrt n}\right)$. However, the computations become more involved, due to the fact that some of random variables that appear are no longer independent.

Denoting $S_{\theta_0^\perp}= S^{n-1}\cap\theta_0^\perp$, we prove the following

\begin{thm}\label{thm:Diagonal}
Let $X$ be a random vector uniformly distributed on $P_{\theta_0^\perp}B_{p}^n$, $p\geq 1$ and write
 $\xi_1=\frac{e_1-e_2+e_3-e_4}{2}$, $\xi_2=\frac{e_1-e_2-e_3+e_4}{2}$, $\overline{\xi}_1=\frac{e_1-e_2}{\sqrt{2}}$, $\overline{\xi}_2=\frac{e_3-e_4}{\sqrt{2}}\in S_{\theta_0^\perp}$. Let $f: S_{\theta_0^\perp}\times S_{\theta_0^\perp}\to\R$ be the function
 $$
 f(\eta_1,\eta_2)=\E\langle X,\eta_1\rangle^2\langle X,\eta_2\rangle^2-\E\langle X,\eta_1\rangle^2\E\langle X,\eta_2\rangle^2.
 $$ For every fixed $p\geq 2$  there exists $n_0(p)\in\N$ such that if $n\geq n_0$ then for every $\eta_1,\eta_2\in S_{\theta_0^\perp}$ such that $\langle\eta_1,\eta_2\rangle=0$ we have that,
$$
f(\xi_1,\xi_2)\leq f(\eta_1,\eta_2)\leq f(\overline{\xi}_1,\overline{\xi}_2),
$$
and for every $1\leq p\leq 2$ there exists $n_1(p)\in\N$ such that if $n\geq n_1$ then
$$
f(\overline{\xi}_1,\overline{\xi}_2)\leq f(\eta_1,\eta_2)\leq f(\xi_1,\xi_2).
$$
\end{thm}

Studying the sign of $f(\overline{\xi}_1,\overline{\xi}_2)$ and $f({\xi}_1,{\xi}_2)$, (see Lemmas 4.8 and 4.10) we obtain the following corollary:

\begin{cor}\label{CorollaryDiagonal}
Let $X$ be a random vector uniformly distributed on $P_{\theta_0^\perp}B_{p}^n$, $p\geq 1$. There exists $n_2(p)\in\N$ such that for all $n\ge n_2$, $X$ satisfies the SNCP with respect to every orthonormal basis in $\theta_0^\perp$.
\end{cor}

 As a consequence of \cite[Proposition 1.9] {AB2}

\begin{cor} Let $X$ be a random vector uniformly distributed on $T(P_{\theta_0^\perp}B_{p}^n)$, $p\geq 1$, $T\colon\R^n\to\R^n$ linear isomorphism. There exists $C(p)>0$ (depending only on $p$) such that $X$ satisfies the General Variance Conjecture with $C=C(p)$.
\end{cor}

\section{Preliminary results}

In this section we will introduce the preliminary results that we need in order to prove Theorems \ref{thm:Bpn} and \ref{thm:Diagonal}. We  briefly review the tools developed in \cite{BN} and \cite{AB3}. Let
$\sigma_p^n$ be the surface measure (Hausdorff measure) on $\partial B_p^n$, the boundary of $B_p^n, p\ge 1$, and denote by $\mu_p^n$
the cone probability measure on $\partial B_p^n$, defined by
$
\mu^n_p(A)=\frac{1}{\textrm{Vol}(B^n_p)}  {\textrm{Vol}(\{ta\in\R^n;a\in A, 0\leq t\leq 1   \}) },\,\,
A\subseteq \partial B^n_p,
$ where Vol denotes the Lebesgue measure.

The following relation between the surface measure and the cone measure on $\partial B_p^n$ was stated in \cite{NR} (see also \cite{AB3}): For almost every point $x\in \partial B^n_p$
$$\frac{d\sigma^n_p(x)}{d\mu_p^n(x)}=n\textrm{Vol}(B^n_p)\left\vert
\nabla(\Vert\cdot\Vert_p)(x)\right\vert.
$$

The cone measure on $\partial B_p^n$ was proved in \cite{SZ} to have the following probabilistic description: Let $g_1,\dots,g_n$ be independent copies of a random variable $g$ with density with respect to the Lebesgue measure given by
$\displaystyle\frac{e^{-|t|^p}}{2\Gamma(1+1/p)}, t\in\R$, $p\ge 1$ and denote
$ S:=\left(\sum_{i=1}^n|g_i|^p\right)^\frac{1}{p}.$ Then
\begin{itemize}
\item The random vector
$\frac{G}{S}:=\left(\dfrac{g_1}{S},\dots,\dfrac{g_n}{S}\right)$  and the random
variable $S$ are independent.
\item $\frac{G}{S}$ is distributed on $\partial B_p^n $ according to the cone
measure $\mu_p^n$.
\end{itemize}

Now, in order to compute the expectation of a suitable function $f(X)$ for $X$ a random vector uniformly distributed on the orthogonal projection of $B_p^n$ onto some hyperplane orthogonal to $\theta\in S^{n-1}$, $P_{\theta^\perp}B_p^n$, we first use Cauchy's formula and pass to an integration on $ \partial B^n_p$ with respect to the surface measure, then use the relation between the surface measure and the cone measure and finally the latter probabilistic representation of the cone measure (see \cite{AB3} for the details). The final result is the starting point for the proof of our main results:

\begin{lemma}\label{lem:ProbabilisticRepresentation}\cite{AB3}
Let $\theta\in S^{n-1}$. If $X$ is a random vector uniformly distributed on $P_{\theta^\perp}B_p^n$, $g_1,\dots,g_n$ are independent copies of $g$ as above and $S=\left(\sum_{i=1}^n|g_i|^p\right)^\frac{1}{p}$, then
for every integrable function $f:P_{\theta^\perp}B_p^n\to\R$
$$
\E f(X)=\frac{\E
f\left(P_{\theta^\perp}\left(\frac{g_1}{S},\dots,\frac{g_n}{S}\right)\right)\left|\sum_{i=1}
^n
\frac{|g_i|^{p-1}}{S^{p-1}}\signum (g_i) \theta_i\right|}{\E
\left|\sum_{i=1}^n
\frac{|g_i|^{p-1}}{S^{p-1}}\signum (g_i) \theta_i\right|}.
$$
where $\theta=(\theta_1,\dots,\theta_n)$ and $\signum (g_i)$ denotes the sign of $g_i$.
\end{lemma}

The following lemma computes the expectation of the random variables involved in terms of the Gamma function:

\begin{lemma}\label{LemmagS}\cite{AB3}\cite{BN}
Let $\alpha\geq 0$, let $g_1,\dots,g_n$ be independent copies of
$g$ as above and $S=\left(\sum_{i=1}^n|g_i|^p\right)^\frac{1}{p}$. Then
$$
\E|g|^\alpha=
\frac{1}{\alpha+1}
\frac{\Gamma\left(1+\frac{\alpha+1}{p}\right)}
{\Gamma\left(1+\frac{1}{p
}\right)}
\qquad\qquad and \qquad\qquad
\E S^\alpha=\frac{n}{n+\alpha}\frac{\Gamma\left(1+
 \frac{n+\alpha}{p}\right)}{\Gamma\left(1+
 \frac{n}{p}\right)}.
$$
\end{lemma}

Our last lemma concerns the so called Gurland's ratio for the Gamma function (see more details in \cite{M}) and it will be crucial in our estimates.

\begin{lemma}\label{lem:GammasSecondTerm}
The function
$$
F(x):=\Gamma(5x)\Gamma(x)/\Gamma(3x)^2
$$
is strictly increasing in $(0,1]$ and verifies $F(\frac{1}{2})=3$.
\end{lemma}

\begin{proof}
The function $F$ is increasing if and only if its logarithm is increasing. Therefore, let us see that the function
$$
h(x)=\log\Gamma(5x)+\log\Gamma(x)-2\log\Gamma(3x)
$$
is increasing. Denoting by $\psi$ the logarithmic derivative of the Gamma function, which verifies (see, for instance \cite{ABR})
$$
\psi(x)=\int_0^\infty\left(\frac{e^{-t}}{t}-\frac{e^{-xt}}{1-e^{-t}}\right)dt,
$$
we have that
\begin{eqnarray*}
h^\prime(x)&=&5\psi(5x)+\psi(x)-6\psi(3x)=5\int_0^\infty\left(\frac{e^{-t}}{t}-\frac{e^{-5xt}}{1-e^{-t}}\right)dt+\cr
&+&\int_0^\infty\left(\frac{e^{-t}}{t}-\frac{e^{-xt}}{1-e^{-t}}\right)dt-6\int_0^\infty\left(\frac{e^{-t}}{t}-\frac{e^{-3xt}}{1-e^{-t}}\right)dt\cr
&=&\int_0^\infty\frac{1}{1-e^{-t}}(-5e^{-5xt}-e^{-xt}+6e^{-3xt})dt.
\end{eqnarray*}
Then,
\begin{eqnarray*}
xh^\prime(x)&=&\int_0^\infty\frac{1}{1-e^{-t}}(-5xe^{-5xt}-xe^{-xt}+6xe^{-3xt})dt\cr
&=&\int_0^\infty\frac{1}{1-e^{-t}}\frac{d}{dt}(e^{-5xt}+e^{-xt}-2e^{-3xt})dt\cr
&=&\int_0^\infty\frac{e^{-t}}{(1-e^{-t})^2}\left((e^{-5xt}-e^{-3xt})-(e^{-3xt}-e^{-xt})\right)dt.\cr
\end{eqnarray*}
Since the function $e^{-y}$ is convex, we have $\displaystyle
\frac{e^{-5xt}-e^{-3xt}}{2xt}\geq\frac{e^{-3xt}-e^{-xt}}{2xt}
$, $\,\forall\, x,t>0$
and so the last integral is positive. Thus, for every $x>0$, $h^\prime(x)>0$ and we obtain the result. It is clear that $F(\frac{1}{2})=3$.
\end{proof}

\section{The SCNP on $B_p^n$}

In this section we will prove Theorem \ref{thm:Bpn}.

\begin{proposition}\label{prop:DecompositionOfF}
Let $X$ be a random vector uniformly distributed on $B_p^n$, $p\geq 1$, and write $\xi_1=\frac{e_1+e_2}{\sqrt{2}}$, $\xi_2=\frac{e_1-e_2}{\sqrt{2}}$. Let $f: S^{n-1}\times S^{n-1}$ be the function
$$
f(\eta_1,\eta_2)=\E\langle X,\eta_1\rangle^2\langle X,\eta_2\rangle^2-\E\langle X,\eta_1\rangle^2\E\langle X,\eta_2\rangle^2.
$$
For every $\eta_1,\eta_2\in S^{n-1}$ such that $\langle\eta_1,\eta_2\rangle=0$, we have
$$
f(\eta_1,\eta_2)= f(e_1,e_2)+2\Big(f(\xi_1,\xi_2)-f(e_1,e_2)\Big)\sum_{i=1}^n \eta_1(i)^2\eta_2(i)^2.
$$
where $\eta_j=(\eta_j(1),\dots \eta_j(n)), j=1,2.$
\end{proposition}

\begin{proof}
Let $B_p^n=P_{e_{n+1}^\perp}B_p^{n+1}$ and let $X$ be a random vector uniformly distributed on $B_p^n$. We first apply Lemma \ref{lem:ProbabilisticRepresentation} to the function $\langle X,\eta_1\rangle^2$. Here, $S=\left(\sum_{i=1}^{n+1}|g_i|^p\right)^\frac{1}{p}$. By independence and symmetry of the $g_i's$, it is straightforward to check that $\E\left\langle G,\eta\right\rangle^2=\E g_1^2$
(it does not depend on $\eta\in S^{n-1}$) and since $\frac{G}{S}$ and $S$ are independent,
$$
\E \langle X,\eta_1\rangle^2=\frac{\E\left\langle\frac{G}{S},\eta_1\right\rangle^2\frac{|g_{n+1}|^{p-1}}{S^{p-1}}}{\E\frac{|g_{n+1}|^{p-1}}{S^{p-1}}}=\frac{\E S^{p-1}\E\left\langle G,\eta_1\right\rangle^2|g_{n+1}|^{p-1}}{\E S^{p+1}\E|g_{n+1}|^{p-1}}
\bigskip
=\frac{\E S^{p-1}\E g_1^2}{\E S^{p+1}}.
$$
The estimates in Lemma 2.2 yield to
$$
\E \langle X,\eta_1\rangle^2\,\E \langle X,\eta_2\rangle^2=\frac{\Gamma\left(\frac{3}{p}\right)^2\Gamma\left(1+\frac{n}{p}\right)^2}{\Gamma\left(\frac{1}{p}\right)^2\Gamma\left(1+\frac{n+2}{p}\right)^2}.
$$
In the same way, we apply Lemma \ref{lem:ProbabilisticRepresentation} to  $\langle X,\eta_1\rangle^2\,\langle X,\eta_2\rangle^2$

\begin{eqnarray*}
\E \langle X,\eta_1\rangle^2\langle X,\eta_2\rangle^2\!\!\!&=&\!\!\!\frac{\E\left\langle\frac{G}{S},\eta_1\right\rangle^2\!\!\left\langle\frac{G}{S},\eta_2\right\rangle^2\frac{|g_{n+1}|^{p-1}}{S^{p-1}}}{\E\frac{|g_{n+1}|^{p-1}}{S^{p-1}}}\cr
\!\!\!&=&\!\!\!\frac{\E S^{p-1}\E\left\langle G,\eta_1\right\rangle^2\!\!\left\langle G,\eta_2\right\rangle^2|g_{n+1}|^{p-1}}{\E S^{p+3}\E|g_{n+1}|^{p-1}}
=\frac{\E S^{p-1}}{\E S^{p+3}}\E\left\langle G,\eta_1\right\rangle^2\!\!\left\langle G,\eta_2\right\rangle^2
\end{eqnarray*}
We compute the last product taking into account  $\eta_1,\eta_2\in S^{n-1}$ and $\langle\eta_1,\eta_2\rangle=0$,

\begin{align*}
&\E\left\langle G,\eta_1\right\rangle^2\left\langle G,\eta_2\right\rangle^2\cr
&\!=\!\E g_1^4\sum_{i=1}^n \eta_1(i)^2\eta_2(i)^2\!\!+\!(\E g_1^2)^2\sum_{i\neq j}\eta_1(i)^2\eta_2(j)^2\!\!+\!2(\E g_1^2)^2\sum_{i\neq j}\eta_1(i)\eta_2(i)\eta_1(j)\eta_2(j)\cr
&\!=\!\E g_1^4\sum_{i=1}^n \eta_1(i)^2\eta_2(i)^2\!\!+\!(\E g_1^2)^2\Big(1-\sum_{i=1}^n\eta_1(i)^2\eta_2(i)^2\Big)-2(\E g_1^2)^2\sum_{i=1}^n\eta_1(i)^2\eta_2(i)^2\cr
&\!=\!(\E g_1^2)^2+\Big(\E g_1^4-3(\E g_1^2)^2\Big)\sum_{i=1}^n \eta_1(i)^2\eta_2(i)^2
\end{align*} and so,
$$
\E \langle X,\eta_1\rangle^2\, \langle X,\eta_2\rangle^2=\frac{(\E g_1^2)^2\E S^{p-1}}{\E S^{p+3}}+\frac{\E S^{p-1}}{\E S^{p+3}}\Big(\E g_1^4-3(\E g_1^2)^2\Big)\sum_{i=1}^n \eta_1(i)^2\eta_2(i)^2.
$$

Notice that the first summand equals
$$
\frac{(\E g_1^2)^2\E S^{p-1}}{\E S^{p+3}}=\frac{\E g_1^2g_2^2|g_{n+1}|^{p-1}\E S^{p-1}}{\E S^{p+3}\E|g_{n+1}|^{p-1}}=\E\langle X,e_1\rangle^2\langle X,e_2\rangle^2.
$$
On the other hand, it is easy to check that
$$
\E g_1^4-3(\E g_1^2)^2=2\left(\E\left(\frac{g_1+g_2}{\sqrt{2}}\right)^2\left(\frac{g_1-g_2}{\sqrt{2}}\right)^2-\E g_1^2g_2^2\right)
$$
and so the factor $\displaystyle \frac{\E S^{p-1}}{\E S^{p+3}}\Big(\E g_1^4-3(\E g_1^2)^2\Big)$ in the second summand is equal to
\begin{align*}
&\frac{2\ \E S^{p-1}}{\E S^{p+3}\E|g_{n+1}|^{p-1}}\left(\E\left(\frac{g_1+g_2}{\sqrt{2}}\right)^2\left(\frac{g_1-g_2}{\sqrt{2}}\right)^2|g_{n+1}|^{p-1}-\E g_1^2g_2^2|g_{n+1}|^{p-1} \right)\cr
&=2\Big(\E \langle X, \xi_1\rangle^2\langle X,\xi_2\rangle^2-\E\langle X,e_1\rangle^2\langle X,e_2\rangle^2\Big).
\end{align*}
The fact that $\E \langle X,\eta\rangle^2$ is independent of $\eta\in S^{n-1}$
finishes the proof.
\end{proof}

\begin{lemma}\label{lem:Lagrange1}
Let $\eta_1,\eta_2\in S^{n-1}$ with $\langle\eta_1,\eta_2\rangle=0$. Then
$$
0\leq \sum_{i=1}^n \eta_1(i)^2\eta_2(i)^2\leq\frac{1}{2}.
$$
The lower bound is attained at any two vectors of the canonical basis. The upper bound is attained at the vectors $\xi_1=\frac{e_i+e_j}{\sqrt{2}}$ and $\xi_2=\frac{e_i-e_j}{\sqrt{2}}$ for any $i\ne j$.
\end{lemma}

\begin{proof}

The lower bound is trivial. For the upper bound consider the function $F\colon\R^{2n}\to\R$ given by $F(\eta_1,\eta_2)=\sum_{i=1}^n\eta_1(i)^2\eta_2(i)^2$ which we want to maximize  under the conditions $\sum_{i=1}^n\eta_1(i)\eta_2(i)=0$ and $\sum_{i=1}^n\eta_1(i)^2=\sum_{i=1}^n\eta_2(i)^2=1$.
Observe that if $(\eta_1,\eta_2)$ is an extremal point so is $(\pm \eta_1,\pm \eta_2)$ and $(\eta_2,\eta_1)$. The proof is a consequence of Lagrange multiplier's theorem.
\end{proof}

\begin{proof}[Proof of Theorem 1.1]

We have seen in the proof of Proposition \ref{prop:DecompositionOfF} that
$$
f(\xi_1,\xi_2)-f(e_1,e_2)=\frac{\E S^{p-1}}{2\ \E S^{p+3}}(\E g_1^4-3(\E g_1^2)^2).
$$
Therefore, its sign is equal to the sign of
$$
\E g_1^4-3(\E g_1^2)^2=\frac{\Gamma\left(\frac{5}{p}\right)}{\Gamma\left(\frac{1}{p}\right)}-3\frac{\Gamma\left(\frac{3}{p}\right)^2}{\Gamma\left(\frac{1}{p}\right)^2}
=\frac{\Gamma\left(\frac{3}{p}\right)^2}{\Gamma\left(\frac{1}{p}\right)^2}\left(F\left(\frac{1}{p}\right)-3\right),
$$
where $\displaystyle
F(x)=\Gamma(5x)\Gamma(x)/\Gamma(3x)^2.
$
By Lemma \ref{lem:GammasSecondTerm} its sign is negative if $p\geq 2$ and positive if $1\leq p\leq 2$.

By Lemma \ref{lem:Lagrange1} and Proposition \ref{prop:DecompositionOfF} the function $f$ attains its maximum (resp. minimum) at $(e_1,e_2)$ and its minimum (resp. maximum) at $(\xi_1,\xi_2)$ depending on whether the sign of $f(\xi_1,\xi_2)-f(e_1,e_2)$ is negative (resp. positive).
\end{proof}

In order to prove Corollary 1.1, we compute the value of $f$ at the extremal pairs,

\begin{lemma}\label{lem:ValueOfF}

$$
f(e_1,e_2)=\frac{\Gamma\left(1+\frac{n}{p}\right)\Gamma\left(\frac{3}{p}\right)^2}{\Gamma\left(1+\frac{n+4}{p}\right)\Gamma\left(\frac{1}{p}\right)^2}\left(1-
\frac{\Gamma\left(1+\frac{n}{p}\right)\Gamma\left(1+\frac{n+4}{p}\right)}{\Gamma
\left(1+\frac{n+2}{p}\right)^2}\right),
$$
and
$$
f(\xi_1,\xi_2)=\frac{\Gamma\left(1+\frac{n}{p}\right)\Gamma\left(\frac{3}{p}\right)^2\left(F\left(\frac{1}{p}\right)-1-2\frac{\Gamma\left(1+\frac{n}{p}\right)\Gamma\left(1+\frac{n+4}{p}\right)}{\Gamma\left(1+\frac{n+2}{p}\right)^2}\right)}{2\Gamma\left(1+\frac{n+4}{p}\right)\Gamma\left(\frac{1}{p}\right)^2}
$$
\end{lemma}
\begin{proof}
We have seen in the proof of Proposition \ref{prop:DecompositionOfF} that
$$
f(e_1,e_2)=\frac{\E S^{p-1}}{\E S^{p+3}}(\E g_1^2)^2-\left(\frac{\E S^{p-1}\E g_1^2}{\E S^{p+1}}\right)^2
$$
and that
$$
f(\xi_1,\xi_2)=\frac{\E S^{p-1}}{2\E S^{p+3}}\left(\E g_1^4-(\E g_1^2)^2\right)-\left(\frac{\E S^{p-1}\E g_1^2}{\E S^{p+1}}\right)^2.
$$
Now substitute the expressions from Lemma \ref{LemmagS}, where $S=\left(\sum_{i=1}^{n+1}|g_i|^p\right)^\frac{1}{p}$.
\end{proof}

\noindent\textsl{Proof of Corollary 1.1.}
Since $n+2=\frac{n+(n+4)}{2}$ and $\log\Gamma(x)$ is strictly convex (\cite{ABR}), $f(e_1,e_2)<0$ for every $p\geq 1$. If $1\leq p<2$, Lemma \ref{lem:GammasSecondTerm} implies $F\left(\frac{1}{p}\right)>3$
and by Stirling's formula (\cite{ABR}, \cite{M}),
$$
\lim_{n\to\infty}\frac{\Gamma\left(1+\frac{n}{p}\right)\Gamma\left(1+\frac{n+4}{p}\right)}{\Gamma\left(1+\frac{n+2}{p}\right)^2}=1.
$$
Thus, for every $1\leq p<2$ there exists $n_0(p)\in\N$ so that if $n\geq n_0$, $f(\xi_1,\xi_2)>0$.
\hfill $\square$

\begin{rmk}
An statement fixing $n$ first yields: For $p=2$ $f(\xi_1,\xi_2)=f(e_1,e_2)<0$ and so, by continuity, for every $n\in\N$ there exists $p_0(n)\in (1,2)$ such that for every $p\ge p_0$ a random vector uniformly distributed on $B_p^n$ satisfies the SNCP with respect to every orthonormal basis.
\end{rmk}

\section{The SNCP on a projection of $B_p^n$}

In this section we will prove Theorem \ref{thm:Diagonal}. The general scheme is analogous to the one used in the previous section. The first Proposition below corresponds to Proposition \ref{prop:DecompositionOfF} for $B_p^n$.

Recall that $\theta_0=\left(\frac{1}{\sqrt n},\dots,\frac{1}{\sqrt n}\right)$ denotes the diagonal direction and $P_{\theta_0^{\perp}}$ the orthogonal projection onto the hyperplane $\theta_0^{\perp}$.

\begin{proposition}\label{prop:DecompositionOfFDiagonal}
	Let $X$ be a random vector uniformly distributed on $P_{\theta_0^{\perp}}B_{p}^n$, $n\geq 4$, $p\geq 1$. Let $f: S_{\theta_0^\perp}\times S_{\theta_0^\perp}\to\R$ be the function
	$$
	f(\eta_1,\eta_2)=\E\langle X,\eta_1\rangle^2\langle X,\eta_2\rangle^2-\E\langle X,\eta_1\rangle^2\E\langle X,\eta_2\rangle^2.
	$$
	and write $\xi_1=\frac{e_1-e_2+e_3-e_4}{2}$, $\xi_2=\frac{e_1-e_2-e_3+e_4}{2}$, $\overline{\xi}_1=\frac{e_1-e_2}{\sqrt{2}}$, $\overline{\xi}_2=\frac{e_3-e_4}{\sqrt{2}}\in S_{\theta_0^\perp}$. For every $\eta_1,\eta_2\in S_{\theta_0^\perp}$ such that $\langle\eta_1,\eta_2\rangle=0$ we have
	$$
	f(\eta_1,\eta_2)= f(\overline{\xi}_1,\overline{\xi}_2)+4\Big(f(\xi_1,\xi_2)-f(\overline{\xi}_1,\overline{\xi}_2)\Big)\sum_{i=1}^n\eta_1(i)^2\eta_2(i)^2.
	$$
\end{proposition}

In order to prove this proposition we first state two lemmas.  Recall $g_1,\dots,g_n$ denote independent copies of a random variable $g$, with density with respect to the Lebesgue measure $\frac{e^{-|t|^p}}{2\Gamma(1+1/p)}$, $G=(g_1,\dots,g_n)$, and $S=\left(\sum_{i=1}^{n}|g_i|^p\right)^\frac{1}{p}$. We define $\psi_{\theta_0} :=\frac{1}{\sqrt n}\left|\sum_{i=1}^n|g_i|^{p-1}\signum(g_i)\right|$.

\begin{lemma}\label{lem:ExpectationScalarProductSquared} For every
$\eta\in S_{\theta_0^\perp}$,
	$\displaystyle \E\langle X,\eta\rangle^2=\frac{\E S^{p-1}\E g_1(g_1-g_2)\psi_{\theta_0}}{\E S^{p+1}\E \psi_{\theta_0}} $ and, in particular, it is independent of $\eta$.
\end{lemma}

\noindent\textit{Proof.} $\eta=(\eta(1),\dots,\eta(n))\in S_{\theta_0^\perp}$ is equivalent to $\sum_{i=1}^n\eta(i)^2=1$, $\sum_{i=1}^n\eta(i)=0.$
	Apply Lemma \ref{lem:ProbabilisticRepresentation} to the function $\langle X,\eta\rangle^2$
	\begin{eqnarray*}
		\E\langle X,\eta\rangle^2&=&\frac{\E\left\langle\frac{G}{S},\eta\right\rangle^2\left|\frac{1}{\sqrt n}\sum_{i=1}^n\frac{|g_i|^{p-1}}{S^{p-1}}\signum (g_i)\right|}{\E\left|\frac{1}{\sqrt n} \sum_{i=1}^n\frac{|g_i|^{p-1}}{S^{p-1}}\signum (g_i)\right|}=\frac{\E S^{p-1}}{\E S^{p+1}}\frac{\E \langle G,\eta\rangle^2\psi_{\theta_0}}{\E\psi_{\theta_0}}\cr
		&=&\frac{\E S^{p-1}}{\E S^{p+1}\E \psi_{\theta_0}}\left(\sum_{i=1}^n \E g_i^2\psi_{\theta_0}\eta(i)^2+\sum_{i\neq j}^n\E g_ig_j\psi_{\theta_0}\eta(i)\eta(j)\right)\cr
		&=&\frac{\E S^{p-1}}{\E S^{p+1}\E \psi_{\theta_0}}\left(\E g_1^2\psi_{\theta_0}\sum_{i=1}^n \eta(i)^2+\E g_1g_2\psi_{\theta_0}\sum_{i\neq j}^n\eta(i)\eta(j)\right)\cr
		&=&\frac{\E S^{p-1}}{\E S^{p+1}\E \psi_{\theta_0}}\left(\E g_1^2\psi_{\theta_0}+\E g_1g_2\psi_{\theta_0}\left(\left(\sum_{i=1}^n\eta(i)\right)^2-1\right)\right).\qquad\qquad\square
	\end{eqnarray*}

In the next lemma we rewrite several expressions in terms of $\displaystyle\sum_{i=1}^n\eta_1(i)^2\eta_2(i)^2$.

\begin{lemma}\label{lem:IdentitiesSums}
Let $\eta_1,\eta_2\in S_{\theta_0^\perp}$ with $\langle\eta_1,\eta_2\rangle=0$. Then
\begin{itemize}
\item $\displaystyle{\sum_{i\neq j}^n\eta_1(i)\eta_1(j)\eta_2(j)^2=-\sum_{i=1}^n\eta_1(i)^2\eta_2(i)^2}$,
\item $\displaystyle{\sum_{i\neq j}\eta_1(i)^2\eta_2(j)^2=1-\sum_{i=1}^n\eta_1(i)^2\eta_2(i)^2}$,
\item $\displaystyle{\sum_{i\neq j}^n\eta_1(i)\eta_2(i)\eta_1(j)\eta_2(j)=-\sum_{i=1}^n\eta_1(i)^2\eta_2(i)^2}$,
\item $\displaystyle{\sum_{i\neq j\neq k}^n\eta_1(i)^2\eta_2(j)\eta_2(k)=-1+2\sum_{i=1}^n\eta_1(i)^2\eta_2(i)^2}$,
\item $\displaystyle{\sum_{i\neq j\neq k}\eta_1(i)\eta_2(i)\eta_1(j)\eta_2(k)=2\sum_{i=1}^n\eta_1(i)^2\eta_2(i)^2}$,
\item $\displaystyle{\sum_{i\neq j\neq k\neq l}^n\eta_1(i)\eta_1(j)\eta_2(k)\eta_2(l)=1-6\sum_{i=1}^n\eta_1(i)^2\eta_2(i)^2}$.
\end{itemize}
\end{lemma}

\begin{proof}
The first three identities are obtained by adding and substracting the sum with $i=j$ and taking into account that $|\eta_1|=|\eta_2|=1$ and $\langle \eta_1,\eta_2\rangle=0$. For the fourth one, notice that since $|\eta_1|=1$,
\begin{eqnarray*}
\sum_{i\neq j\neq k}^n\!\eta_1(i)^2\eta_2(j)\eta_2(k)\!\!\! &\!\!\!\!=\!\!\!\!\!&\sum_{k\neq j}\Big(\eta_2(j)\eta_2(k)-\eta_1(j)^2\eta_2(j)\eta_2(k)-\eta_1(k)^2\eta_2(j)\eta_2(k)\Big)\cr
\!\!\!\!\!&=&\!\!\!-\sum_{j=1}^n\eta_2(j)^2-\!\sum_{j\neq k}^n\Big(\eta_1(j)^2\eta_2(j)\eta_2(k)+ \eta_1(k)^2\eta_2(j)\eta_2(k)\Big)
\end{eqnarray*}
and then use the first identity. For the fifth one, notice that since $\langle\eta_1,\eta_2\rangle=0$
$$
\sum_{i\neq j\neq k}^n\eta_1(i)\eta_2(i)\eta_1(j)\eta_2(k) =\sum_{k\neq j}\Big(0-\eta_1(j)^2\eta_2(j)\eta_2(k)-\eta_1(j)\eta_1(k)\eta_2(k)^2\Big)
$$
and then use the first identity. For the last one, we use $\sum_{i=1}^n\eta_1(i)=0 $
\begin{align*}
&\sum_{i\neq j\neq k\neq l}^n\eta_1(i)\eta_1(j)\eta_2(k)\eta_2(l) \cr
&=\sum_{j\neq k\neq l}\Big(0-\eta_1(j)^2\eta_2(k)\eta_2(l)-\eta_1(j)\eta_1(k)\eta_2(k)\eta_2(l)-\eta_1(j)\eta_2(k)\eta_1(l)\eta_2(l)\Big)
\end{align*}
and then use the fourth and the fifth identities.
\end{proof}

\begin{proof}[Proof of Proposition \ref{prop:DecompositionOfFDiagonal}]
 By Lemma \ref{lem:ProbabilisticRepresentation} we have that for every  $\eta_1,\eta_2\in S_{\theta_0^\perp}$, with $\langle\eta_1,\eta_2\rangle=0$
\begin{eqnarray*}
\E\langle X,\eta_1\rangle^2\langle X,\eta_2\rangle^2&=&\frac{\E\left\langle\frac{G}{S},\eta_1\right\rangle^2\left\langle\frac{G}{S},\eta_2\right\rangle^2\left|\frac{1}{\sqrt n}\sum_{i=1}^n\frac{|g_i|^{p-1}}{S^{p-1}}\signum (g_i)\right|}{\E\left|\frac{1}{\sqrt{n}}\sum_{i=1}^n\frac{|g_i|^{p-1}}{S^{p-1}}\signum (g_i)\right|}\medskip\cr
&=&\frac{\E S^{p-1}}{\E S^{p+3}}\frac{\E \langle G,\eta_1\rangle^2\langle G,\eta_2\rangle^2\psi_{\theta_0}}{\E\psi_{\theta_0}}.\cr
\end{eqnarray*}
Expanding the product and since the $g_i's$ are identically distributed, we have
\begin{eqnarray*}
&&\E \langle G,\eta_1\rangle^2\langle G,\eta_2\rangle^2\psi_{\theta_0}=\sum_{i,j,k,l=1}^n\E g_ig_jg_kg_l\psi_{\theta_0}\eta_1(i)\eta_1(j)\eta_2(k)\eta_2(l)\cr
&=&\E g_1^4\psi_{\theta_0}\sum_{i=1}^n\eta_1(i)^2\eta_2(i)^2\cr
&+&\E g_1^3g_2\psi_{\theta_0}\Big(2\sum_{i\neq j}^n\eta_1(i)\eta_1(j)\eta_2(j)^2+2\sum_{i\neq j}^n\eta_1(i)^2\eta_2(i)\eta_2(j)\Big)\cr
&+&\E g_1^2g_2^2\psi_{\theta_0}\Big(\sum_{i\neq j}^n\eta_1(i)^2\eta_2(j)^2+2\sum_{i\neq j}^n\eta_1(i)\eta_2(i)\eta_1(j)\eta_2(j)\Big)\cr
&+&\E g_1^2g_2g_3\psi_{\theta_0}\Big(\sum_{i\neq j\neq k}^n\eta_1(i)^2\eta_2(j)\eta_2(k)+4\sum_{i\neq j\neq k}^n\eta_1(i)\eta_2(i)\eta_1(j)\eta_2(k)\cr
&&\qquad\qquad\qquad\qquad\qquad\qquad\qquad\quad\,\,\, +\sum_{i\neq j\neq k}\eta_2(i)^2\eta_1(j)\eta_1(k)\Big)\cr
&+&\E g_1g_2g_3g_4\psi_{\theta_0}\sum_{i\neq j\neq k\neq l}\eta_1(i)\eta_1(j)\eta_2(k)\eta_2(l).
\end{eqnarray*}
and by the identities in Lemma \ref{lem:IdentitiesSums} we obtain
\begin{eqnarray*}
&&\E\langle G,\eta_1\rangle^2\langle G,\eta_2\rangle^2=\E\big(g_1^2g_2^2-2g_1^2g_2g_3+g_1g_2g_3g_4\big)\psi_{\theta_0}\cr
&+&\left(\E\big(g_1^4-4g_1^3g_2-3g_1^2g_2^2+12g_1^2g_2g_3-6g_1g_2g_3g_4\big)\psi_{\theta_0}\right)\sum_{i=1}^n\eta_1(i)^2\eta_2(i)^2.
\end{eqnarray*}

We can express the first summand as
$$
\E(g_1^2g_2^2-2g_1^2g_2g_3+g_1g_2g_3g_4)\psi_{\theta_0}=\frac{1}{4}\E(g_1-g_2)^2(g_3-g_4)^2\psi_{\theta_0}=\E\langle G,\overline{\xi}_1\rangle^2\langle G,\overline{\xi}_2\rangle^2\psi_{\theta_0}.
$$
and the factor $\displaystyle \E(g_1^4-4g_1^3g_2-3g_1^2g_2^2+12g_1^2g_2g_3-6g_1g_2g_3g_4)\psi_{\theta_0}$ in the second summand as
$$4\E\left(\frac{g_1-g_2+g_3-g_4}{{2}}\right)^2\!\left(\frac{g_1-g_2-g_3+g_4}{{2}}\right)^2\!\psi_{\theta_0}-4\E\left(\frac{g_1-g_2}{\sqrt{2}}\right)^2\!\left(\frac{g_3-g_4}{\sqrt{2}}\right)^2\!\psi_{\theta_0}$$
$$=4\E\langle G,\xi_1\rangle^2\langle G,\xi_2\rangle^2\psi_{\theta_0}-4\E\langle G,\overline{\xi}_1\rangle^2\langle G,\overline{\xi}_2\rangle^2\psi_{\theta_0}
$$
Consequently,
\begin{eqnarray*}
\E\langle X,\eta_1\rangle^2\langle X,\eta_2\rangle^2&=&\E\langle X,\overline{\xi}_1\rangle^2\langle X,\overline{\xi}_2\rangle^2\cr
&+&4\left(\E\langle X,\xi_1\rangle^2\langle X,\xi_2\rangle^2-\E\langle X,\overline{\xi}_1\rangle^2\langle G,\overline{\xi}_2\rangle^2\right)\sum_{i=1}^n\eta_1(i)^2\eta_2(i)^2
\end{eqnarray*}
and,  since by Lemma \ref{lem:ExpectationScalarProductSquared} the value of $\E \langle X,\eta\rangle^2$ does not depend on the vector $\eta\in S_{\theta_0^\perp}$ we obtain the result.
\end{proof}

The following lemma is analogous to Lemma \ref{lem:Lagrange1}

\begin{lemma}\label{lem:EstimateSumDiagonal}
Let $\eta_1,\eta_2\in S_{\theta_0^\perp}$ with $\langle\eta_1,\eta_2\rangle=0$, $n\ge 4$. Then

$$ 0\leq\sum_{i=1}^n\eta_1(i)^2\eta_2(i)^2\leq\frac{1}{4}.
$$
The lower bound is attained at the vectors $\overline{\xi}_1=\frac{e_1-e_2}{\sqrt{2}}$, $\overline{\xi}_2=\frac{e_3-e_4}{\sqrt{2}}$. The upper bound is attained at the vectors $\xi_1=\frac{e_1-e_2+e_3-e_4}{2}$ and $\xi_2=\frac{e_1-e_2-e_3+e_4}{\sqrt{2}}$.
\end{lemma}

\begin{proof}

The lower bound is trivial. For the upper bound consider the function $F\colon\R^{2n}\to\R$ given by $F(\eta_1,\eta_2)=\sum_{i=1}^n\eta_1(i)^2\eta_2(i)^2$ which we want to maximize under the conditions $\sum_{i=1}^n\eta_1(i)=\sum_{i=1}^n\eta_2(i)=\sum_{i=1}^n\eta_1(i)\eta_2(i)=0$ and $\sum_{i=1}^n\eta_1(i)^2=\sum_{i=1}^n\eta_2(i)^2=1$.
Observe that if $(\eta_1,\eta_2)$ is an extremal point so is $(\pm \eta_1,\pm \eta_2)$ and $(\eta_2,\eta_1)$.  Applying Lagrange multiplier's theorem, there exist $A,B,C\in\R$ such that the extremal points verify
$$ \eta_1(i)\eta_2(i)^2-A\eta_1(i)-B\eta_2(i)-C=0\qquad\forall\ i=1\dots n$$
and, by the observation above, also verify the equality exchanging $\eta_j$ and $\pm\eta_j$ $(j=1,2)$ and $\eta_2$ and $\eta_1$.

That implies $B=C=0$ and $\eta_1(i)^2=\eta_2(i)^2\ \forall\ i=1\dots n$. Write $k$,  $0\le k\le n$, for the number of non zero coordinates of $\eta_1$ (or $\eta_2$).  Since $\sum_{i=1}^n\eta_1(i)^2\eta_2(i)^2=A$ we have $k\eta_1(i)^2=k\eta_2(i)^2=A$ for every non zero coordinate.  $k=0,1,2,3$ do not verify the conditions, so the maximum value is attained at $k=4$ and corresponds to the vectors $\xi_1$ and $\xi_2$.\end{proof}

We now proceed to determining the sign of $f(\xi_1,\xi_2)-f(\overline{\xi}_1,\overline{\xi}_2)$.  We will use the following probabilistic argument:

\begin{lemma}\label{lem:GeneralProbabilisticLemma}
Denote $Y=(g_1,\dots ,g_k)$ and  $Y_k=\sum_{i=1}^k \signum (g_i) |g_i|^{p-1}$. Let $Z$  be a symmetric real variable independent of $Y$ and  $h:\R^k\to\R$ integrable. Then
$$
\E h(g_1,\dots, g_k)|Y_k+Z|=\E h(g_1,\dots, g_k)\E|Z|+\E h(g_1,\dots, g_k)(|Y_k|-|Z|)\chi_{\{|Y_k|\ge |Z|\}}
$$
and  $$
\Big|\E h(g_1,\dots, g_k)(|Y_k|-|Z|)\chi_{\{|Y_k|\ge |Z|\}}\Big|\le \left(\E |h(g_1,\dots, g_k)|^2\right)^{1/2} \left(\E |Y_k|^2\right)^{1/2}.
$$
\end{lemma}

\begin{proof}

Write $ h=h(g_1,\dots, g_k)$ for short. Our hypothesis readily imply, $$\E h\cdot|Y_k+Z|=\E h\cdot\frac{1}{2}(|Y_k+Z|+|Y_k-Z|)=
\E h\cdot \max\{|Y_k|,|Z|\}$$

Fix $(g_1,\dots ,g_k)$ and compute the expectation with respect to $Z$. We have,
\begin{eqnarray*}
&\E_{Z}&\!\!\! h(g_1,\dots, g_k)\cdot \max\{|Y_k|,|Z|\}=h(g_1,\dots, g_k)\cdot\!\int_0^{\infty}\!\Pro_{Z}\{\max\{|Y_k|,|Z|\}>t\}\, dt\cr
&=&\!\!\! h\cdot\big(|Y_k|+\int_{|Y_k|}^{\infty}\Pro_{Z}\{|Z|>t\}\, dt\big)=h\cdot \big(|Y_k|-\int_{0}^{|Y_k|}\!\Pro_{Z}\{|Z|>t\}\, dt\big)+h\cdot \E |Z|\cr
&=&\!\!\!
h\cdot \E |Z|+h\cdot \int_{0}^{|Y_k|}\Pro_{Z}\{|Z|\le t\}\, dt.\cr
\end{eqnarray*}

Finally, notice that
$$
\int_{0}^{|Y_k|}\Pro_{Z}\{|Z|\le t\}\, dt=|Y_k|\int_0^1\Pro_Z\{|Z|\leq t|Y_k|\}\, dt=|Y_k| \E_Z\left( 1-\min\left\{1,\frac{|Z|}{|Y_k|}\right\}\right)
$$
Taking now expectation $\E_Y$ to finish the proof of the first statement. For the second one notice that $\left|\E h(g_1,\dots, g_k)(|Y_k|-|Z|)\chi_{\{|Y_k|\ge |Z|\}}\right|\le \E |h(g_1,\dots, g_k)|\cdot|Y_k|$ and use Cauchy-Schwarz inequality.
\end{proof}

The following estimate will be useful in the sequel,

\begin{lemma}\label{lem:EstimatePsi}\cite{AB3}
	For some absolute constants $c,C>0$,
	 $$\frac{c}{\sqrt p}\le\E\psi_{\theta_0}\le\frac{C}{\sqrt p}\qquad  if\qquad  1\le p\leq n.$$
	
\end{lemma}

We have the following result regarding the sign of $f(\xi_1,\xi_2)-f(\overline{\xi}_1,\overline{\xi}_2)$, which shall give Theorem \ref{thm:Diagonal} as a consequence.

\begin{proposition}\label{prop:PropositionSignDiagonal}
Let $X$ be a random vector uniformly distributed on $P_{\theta_0}B_{p}^n$ and let $f: S_{\theta_0^\perp}\times S_{\theta_0^\perp}\to\R$ given by
$$
f(\eta_1,\eta_2)=\E\langle X,\eta_1\rangle^2\langle X,\eta_2\rangle^2-\E\langle X,\eta_1\rangle^2\E\langle X,\eta_2\rangle^2.
$$
Write $\xi_1=\frac{e_1-e_2+e_3-e_4}{2}$, $\xi_2=\frac{e_1-e_2-e_3+e_4}{2}$, $\overline{\xi}_1=\frac{e_1-e_2}{\sqrt{2}}$, $\overline{\xi}_2=\frac{e_3-e_4}{\sqrt{2}}\in S_{\theta_0^\perp}$.
For every $p>2$ there exists $n_0(p)\in\N$ such that if $n\geq n_0$,
$$
f(\xi_1,\xi_2)-f(\overline{\xi}_1,\overline{\xi}_2)\leq 0
$$
and for every $1\leq p\leq 2$ there exists $n_1(p)\in\N$ such that if $n\geq n_1$,
$$
f(\xi_1,\xi_2)-f(\overline{\xi}_1,\overline{\xi}_2)\geq 0.
$$
\end{proposition}

\begin{proof}

Since $f$ is constant for $p=2$, we will focus on the cases $p\ne 2$. We have seen in the proof of Proposition \ref{prop:DecompositionOfFDiagonal} that
$$
4(f(\xi_1,\xi_2)-f(\overline{\xi}_1,\overline{\xi}_2))=\frac{\E S^{p-1}}{\E S^{p+3}\E\psi_{\theta_0}}\E(g_1^4-4g_1^3g_2-3g_1^2g_2^2+12g_1^2g_2g_3-6g_1g_2g_3g_4)\psi_{\theta_0}
$$
and so, the sign of $f(\xi_1,\xi_2)-f(\overline{\xi}_1,\overline{\xi}_2)$ is equal to the sign of $\E(h\ \psi_{\theta_0})$   where $h(g_1,g_2,g_3,g_4)=g_1^4-4g_1^3g_2-3g_1^2g_2^2+12g_1^2g_2g_3-6g_1g_2g_3g_4.$

We apply then Lemma \ref{lem:GeneralProbabilisticLemma} to
$\displaystyle{Y_4=\sum_{i=1}^4\signum(g_i)|g_i|^{p-1}}$,
$\displaystyle{Z=\sum_{i=5}^n\signum(g_i)|g_i|^{p-1}}$ and $h$ as above and $$
\sqrt{n}\psi_{\theta_0}=\left|\sum_{i=1}^4\signum(g_i)|g_i|^{p-1}+\sum_{i=5}^n\signum(g_i)|g_i|^{p-1}\right|=|Y_4+Z|,
$$

On one hand,
$$\E h(g_1,\dots, g_4)\E|Z|= \frac{\Gamma\left(\frac{3}{p}\right)^2}{\Gamma\left(\frac{1}{p}
\right)^2}\left(F\left(\frac{1}{p}\right)-3\right)
\E\left|\sum_{i=5}^n|g_i|^{p-1}\signum (g_i)\right|$$
where $F(x)=\Gamma(5x)\Gamma(x)/\Gamma(3x)^2$
since, the $g_i's$ are i.i.d. symmetric random variables and by the computation in the proof of Theorem \ref{thm:Bpn}

$$\E h(g_1,g_2,g_3,g_4)=\E g_1^4-3(\E g_1^2)^2
=\frac{\Gamma\left(\frac{3}{p}\right)^2}{\Gamma\left(\frac{1}{p}
\right)^2}\left(F\left(\frac{1}{p}\right)-3\right)$$

Also, by Lemma \ref{lem:EstimatePsi} we have
$$\frac{c\sqrt n}{\sqrt p}\le \E\left|\sum_{i=5}^n|g_i|^{p-1}\signum (g_i)\right|\le\frac{C\sqrt n}{\sqrt p}$$ provided that $1\le p\le n$ for some absolute constants $c,C>0$.

On the other hand, by straightforward computations
$$\E |h(g_1,\dots, g_4)|^2=\E g_1^8+10\E g_1^6\E g_1^2+9(\E g_1^4)^2+144\E g_1^4(\E g_1^2)^2+36(\E g_1^2)^4$$  which, by Lemma \ref{LemmagS} is bounded by an absolute constant. Thus,
$$\left(\E |h(g_1,\dots, g_4)|^2\right)^{1/2}\E ( |Y_4|^2)^{1/2}\le C\left(\E|g_1|^{2p-2}\right)^\frac{1}{2}\le
\frac{C'}{\sqrt p}
$$
again by Lemma \ref{LemmagS} for some absolute constant $C'>0$.

Recall that factor $F\left(\frac{1}{p}\right)-3$ is positive for $1\le p<2$ and negative for $p>2$. We put all estimates together and obtain for some absolute constants $c_1,c_2,C_1,C_2>0$:\bigskip

Let $p>2$, then for  $n\ge n_0(p)$
$$\sqrt{n}\ \E\ h(g_1,g_2,g_3,g_4)\psi_{\theta_0}\leq\frac{ C_1\sqrt n\left(\frac{\Gamma\left(\frac{3}{p}\right)^2}{\Gamma\left(\frac{1}{p}\right)^2}\left(F\left(\frac{1}{p}\right)-3\right)\right)+C_2}{\sqrt{ p}}<0
$$

Let $1\leq p<2$, then  for  $n\ge n_1(p)$
 $$
\sqrt{n}\E h(g_1,g_2,g_3,g_4)\psi_{\theta_0}\geq\frac{ c_1\sqrt n\left(\frac{\Gamma\left(\frac{3}{p}\right)^2}{\Gamma\left(\frac{1}{p}\right)^2}\left(F\left(\frac{1}{p}\right)-3\right)\right)-c_2}{\sqrt{ p}}>0
$$\end{proof}

\begin{proof}[Proof of Theorem 1.2]

Proceed as in the proof of Theorem \ref{thm:Bpn}, using Lemma 4.3 and Proposition \ref{prop:DecompositionOfFDiagonal}.
\end{proof}

Finally, in order to deduce Corollary \ref{CorollaryDiagonal} we shall compute the sign of $f(\overline{\xi}_1,\overline{\xi}_2)$ for $p\ge 2$ and $f({\xi}_1,{\xi}_2)$ for $1\le p\le 2$. For that matter, we denote $ \overline{g}_1,\dots\overline{g}_n$ i.i.d. copies of
$g$ and the $g_i's$, and ${\overline{\psi}_{\theta_0}=\left|\frac{1}{\sqrt n}\sum_{i=1}^n\signum(\overline{g}_i)|\overline{g}_i|^{p-1}\right|}$.

\begin{lemma}\label{lem:MaximizingPairDiagonal}

Let $X$ a be random vector uniformly distributed on $P_{\theta_0^{\perp}}B_{p}^n$, $p\ge 1$ and $f: S_{\theta_0^\perp}\times S_{\theta_0^\perp}\to\R$ given by
$$
f(\eta_1,\eta_2)=\E\langle X,\eta_1\rangle^2\langle X,\eta_2\rangle^2-\E\langle X,\eta_1\rangle^2\E\langle X,\eta_2\rangle^2.
$$
Write $\overline{\xi}_1=\frac{e_1-e_2}{\sqrt{2}}$, $\overline{\xi}_2=\frac{e_3-e_4}{\sqrt{2}}\in S_{\theta_0^\perp}$. Then
$$
f(\overline{\xi}_1,\overline{\xi}_2)=\frac{\E S^{p-1}\E h(g_1,g_2,g_3,g_4,\overline{g}_1,\overline{g}_2)\psi_{\theta_0}\overline{\psi}_{\theta_0}}{\E S^{p+3}\E \psi_{\theta_0}\overline{\psi}_{\theta_0}},
$$
where $h:\R^6\to\R$ is defined by
\begin{eqnarray*}
h(x_1,\dots,x_6)&=& x_1x_2^2-2x_1^2x_2x_3+x_1x_2x_3x_4\cr &-&\frac{\E S^{p-1}\E S^{p+3}}{(\E S^{p+1})^2}(x_1^2x_5^2-2x_1x_2x_5^2+x_1x_2x_5x_6)\cr
\end{eqnarray*}
\end{lemma}
\begin{proof}
We have seen in the proof of Proposition \ref{prop:DecompositionOfFDiagonal} that
\begin{eqnarray*}
\E\langle X,\overline{\xi}_1\rangle^2\langle X,\overline{\xi}_2\rangle^2&=&\frac{\E S^{p-1}\E(g_1g_2^2-2g_1^2g_2g_3+g_1g_2g_3g_4)\psi_{\theta_0}}{\E S^{p+3}\E \psi_{\theta_0}}\cr
&=&\frac{\E S^{p-1}\E(g_1g_2^2-2g_1^2g_2g_3+g_1g_2g_3g_4)\psi_{\theta_0}\overline{\psi}_{\theta_0}}{\E S^{p+3}\E \psi_{\theta_0}\overline{\psi}_{\theta_0}}.\cr
\end{eqnarray*}
and in the proof of Lemma \ref{lem:ExpectationScalarProductSquared},
$$
\E\langle X,\overline{\xi}_1\rangle^2\E\langle X,\overline{\xi}_2\rangle^2=\frac{(\E S^{p-1})^2\E (g_1^2-g_1g_2)(\overline{g}_1^2-\overline{g}_1\overline{g}_2)\psi_{\theta_0}\overline{\psi}_{\theta_0}}{(\E S^{p+1})^2\E \psi_{\theta_0}\overline{\psi}_{\theta_0}}.
$$
Since $
\E g_1g_2\overline{g}_1^2\psi_{\theta_0}\overline{\psi}_{\theta_0}=\E\overline{g}_1\overline{g}_2g_1^2\psi_{\theta_0}\overline{\psi}_{\theta_0},
$
we obtain the result.
\end{proof}

Therefore, the sign of $f(\overline{\xi}_1,\overline{\xi}_2)$ coincides with the sign of $\E h\psi_{\theta_0}\overline{\psi}_{\theta_0}$. We split the latter quantity in four terms by using Lemma \ref{lem:GeneralProbabilisticLemma}.

\begin{lemma} Denote $\displaystyle{\overline{Y}_2=\sum_{i=1}^2|\overline{g}_i|^{p-1}\signum(\overline{g}_i)}, {Y_4=\sum_{i=1}^4|g_i|^{p-1}\signum(g_i)}$ and  $\displaystyle \overline{Z}=\sum_{i=3}^n|\overline{g}_i|^{p-1}\signum(\overline{g}_i), {Z=\sum_{i=5}^n|g_i|^{p-1}\signum(g_i)}$.
	Then,
\begin{align*}
&n\ \E h(g_1,g_2,g_3,g_4,\overline{g}_1,\overline{g}_2)\psi_{\theta_0}\overline{\psi}_{\theta_0}=\E h(g_1,g_2,g_3,g_4,\overline{g}_1,\overline{g}_2)\E\left|Z\right|\E\left|\overline{Z}\right|\cr
&+\E\left|\overline{Z}\right|\E h(g_1,g_2,g_3,g_4,\overline{g}_1,\overline{g}_2)
\left(\left|Y_4\right|-\left|Z\right|\right)\chi_{\{\left|Y_4\right|\geq\left|Z\right|\}}\cr
&+\E\left|Z\right|\E h(g_1,g_2,g_3,g_4,\overline{g}_1,\overline{g}_2)
\left(\left|\overline{Y_2}\right|-\left|\overline{Z}\right|\right)\chi_{\{\left|\overline{Y_2}\right|\geq\left|\overline{Z}\right|\}}\cr
&+\E h(g_1,g_2,g_3,g_4,\overline{g}_1,\overline{g}_2)\left(\left|Y_4\right|-\left|Z\right|\right)\chi_{\{\left|Y_4\right|\geq\left|Z\right|\}}\left(\left|\overline{Y_2}\right|-\left|\overline{Z}\right|\right)\chi_{\{\left|\overline{Y_2}\right|\geq\left|\overline{Z}\right|\}}.
\end{align*}
\end{lemma}

\begin{proof}
First condition on the random variables $g_1,\dots g_n$ and apply Lemma \ref{lem:GeneralProbabilisticLemma} with $\overline{Y_2}$ and $\overline{Z}$.
Then take expectations with respect to  $g_1,\dots g_n$, use Fubini's theorem and, conditioning on  $\overline{g}_1,\dots \overline{g}_n$, apply again Lemma  with
$Y_4$ and $Z$.
\end{proof}

\begin{lemma} Let $p\ge 1$ and $\overline{\xi}_1=\frac{e_1-e_2}{\sqrt{2}}$, $\overline{\xi}_2=\frac{e_3-e_4}{\sqrt{2}}\in S_{\theta_0^\perp}$.  For every $n\ge n_0(p)$ for some $n_0(p)\in\N$,
	$$f(\overline{\xi}_1,\overline{\xi}_2)<0$$

\end{lemma}
\begin{proof}

We proceed as in the proof of Proposition \ref{prop:PropositionSignDiagonal}:

By Lemma \ref{lem:MaximizingPairDiagonal}, we will compute the sign of $n \E h(g_1,g_2,g_3,g_4,\overline{g}_1,\overline{g}_2)\psi_{\theta_0}\overline{\psi}_{\theta_0}$. For that matter we apply Lemma 4.7 and estimate each summand.
$$ \E\ h=
-\frac{\E S^{p-1}\E S^{p+3}}{(\E S^{p+1})^2}(\E g_1^2)^2=-\frac{\Gamma\left(1+\frac{n-1}{p}\right)\Gamma\left(1+\frac{n+3}{p}\right)\Gamma\left(\frac{3}{p}\right)^2}{\Gamma\left(1+\frac{n+1}{p}\right)^2\Gamma\left(\frac{1}{p}\right)^2}<-c$$
by definition of $h$ and Lemma \ref{LemmagS}, for some absolute constant $c>0$. Also, $\E h^2(g_1,g_2,g_3,g_4,\overline{g}_1,\overline{g}_2)\le C.$ In the sequel we shall use the same letter $c,C...$ to denote possibly different values of an absolute constant $c,C...>0$.

According to Lemma \ref{lem:GeneralProbabilisticLemma} and Lemma \ref{LemmagS}, the second summand has absolute value bounded by
$$\E|\overline{Z}|(\E h^2)^{1/2}(\E |\overline{Y_2}|^2)^{1/2}\le \frac{C}{\sqrt p}\E|\overline{Z}|$$
and in the same way, the third summand has absolute value bounded by $$\E|{Z}|(\E h^2)^{1/2}(\E |Y_4|^2)^{1/2}\le \frac{C}{\sqrt p}\E|{Z}|$$
Similarly, the forth summand has absolute value bounded by $\frac{C}{p}$ and finally, Lemma \ref{lem:EstimatePsi} implies
$\displaystyle
c\frac{\sqrt{n}}{\sqrt{p}}\le\E\left|Z\right|, \E\left|\overline{Z}\right|\le C\frac{\sqrt{n}}{\sqrt{p}}$ whenever $p\leq n$.

We put all estimates together and conclude that for $1\leq p\leq n$ and some absolute constants:
	$$
	\frac{np\E S^{p+3}\E \psi_{\theta_0}\overline{\psi}_{\theta_0}}{\E S^{p-1}}f(\overline{\xi}_1,\overline{\xi}_2)\leq-C_1n+C_2\sqrt{n}+C_3\sqrt{n}+C_4.
	$$
	
 The result now easily follows.
	\end{proof}

In the following lemma we compute the value of $f(\xi_1,\xi_2)$.

\begin{lemma}\label{lem:MaximizingPairDiagonalp<2}

Let $X$ be a random vector uniformly distributed on $P_{\theta_0^\perp}B_{p}^n$, $p\geq 1$, and $f: S_{\theta_0^\perp}\times S_{\theta_0^\perp}\to\R$ given by
$$
f(\eta_1,\eta_2)=\E\langle X,\eta_1\rangle^2\langle X,\eta_2\rangle^2-\E\langle X,\eta_1\rangle^2\E\langle X,\eta_2\rangle^2.
$$
Write $\xi_1=\frac{e_1-e_2+e_3-e_4}{2}$, $\xi_2=\frac{e_1-e_2-e_3+e_4}{2}$.  Then
$$
f(\xi_1,\xi_2)=\frac{\E S^{p-1}\E h(g_1,g_2,g_3,g_4,\overline{g}_1,\overline{g}_2)\psi_{\theta_0}\overline{\psi}_{\theta_0}}{\E S^{p+3}\E \psi_{\theta_0}\overline{\psi}_{\theta_0}},
$$
where $h:\R^6\to\R$ is defined by
\begin{eqnarray*}
h(x_1,x_2,x_3,x_4,x_5,x_6)&=&\frac{1}{4}x_1^4-x_1^3x_2-\frac{3}{4}x_1^2x_2^2+x_1^2x_2x_3-\frac{1}{2}x_1x_2x_3x_4+x_1x_2^2\cr
&-&\frac{\E S^{p-1}\E S^{p+3}}{(\E S^{p+1})^2}(x_1^2x_5^2-2x_1x_2x_5^2+x_1x_2x_5x_6).
\end{eqnarray*}
\end{lemma}

\begin{proof}
We have seen in the proof of Proposition \ref{prop:DecompositionOfFDiagonal} that
\begin{eqnarray*}
\E(g_1^4-4g_1^3g_2-3g_1^2g_2^2&+&12g_1^2g_2g_3-6g_1g_2g_3g_4)
\psi_{\theta_0}=\cr
&=&4\E\langle G,\xi_1\rangle^2\langle G,\xi_2\rangle^2\psi_{\theta_0}-4\E\langle G,\overline{\xi}_1\rangle^2\langle G,\overline{\xi}_2\rangle^2\psi_{\theta_0}
\end{eqnarray*}
and then, taking into account that
$$
\E\langle G,\overline{\xi}_1\rangle^2\langle G,\overline{\xi}_2\rangle^2\psi_{\theta_0}=\E(g_1g_2^2-2g_1^2g_2g_3+g_1g_2g_3g_4)\psi_{\theta_0}
$$
we obtain that
\begin{align*}
&\E\langle X,\xi_1\rangle^2\langle X,\xi_2\rangle^2=\frac{\E S^{p-1}}{\E S^{p+3}\E\psi_{\theta_0}}\E\langle G,\xi_1\rangle^2\langle G,\xi_2\rangle^2\psi_{\theta_0}\cr
&=\frac{\E S^{p-1}}{\E S^{p+3}\E\psi_{\theta_0}\overline{\psi}_{\theta_0}}\Big(\E(\frac{1}{4}g_1^4-g_1^3g_2-\frac{3}{4}g_1^2g_2^2+g_1^2g_2g_3-\frac{1}{2}g_1g_2g_3g_4+g_1g_2^2)\psi_{\theta_0}\overline{\psi}_{\theta_0}\Big)
\end{align*}
Since, by Lemma \ref{lem:ExpectationScalarProductSquared}
$$
\E\langle X,\xi_1\rangle^2\E\langle X,\xi_2\rangle^2=\frac{(\E S^{p-1})^2\E (g_1^2-g_1g_2)(\overline{g}_1^2-\overline{g}_1\overline{g}_2)\psi_{\theta_0}\overline{\psi}_{\theta_0}}{(\E S^{p+1})^2\E \psi_{\theta_0}\overline{\psi}_{\theta_0}}
$$
and
$\E g_1g_2\overline{g}_1^2\psi_{\theta_0}\overline{\psi}_{\theta_0}=\E\overline{g}_1\overline{g}_2g_1^2\psi_{\theta_0}\overline{\psi}_{\theta_0},
$
we obtain the result.
\end{proof}

\begin{lemma}
Let $\xi_1=\frac{e_1-e_2+e_3-e_4}{2}$, $\xi_2=\frac{e_1-e_2-e_3+e_4}{2}\in S_{\theta_0^\perp}$. There exists $n_0\in\N$ such that for every $n\ge n_0$ and $1\le p\le 2$
	$$f(\xi_1,\xi_2)<0$$

\end{lemma}
\begin{proof}

By Lemma 4.8 we must compute the sign of $\E h(g_1,g_2,g_3,g_4,\overline{g}_1,\overline{g}_2)\psi_{\theta_0}\overline{\psi}_{\theta_0}$. For that matter we apply, Lemma 4.7 with the same choice of random variables and $h$ as above. The behaviour of the sign is determined by the sign of

\begin{eqnarray*}
\E h
&=&\frac{\Gamma\left(\frac{3}{p}\right)^2}{4\Gamma\left(\frac{1}{p}\right)^2}\left(F\left(\frac{1}{p}\right)-3-\frac{4\Gamma\left(1+\frac{n-1}{p}\right)\Gamma\left(1+\frac{n+3}{p}\right)}{\Gamma\left(1+\frac{n+1}{p}\right)^2}\right)\cr
&\leq&\frac{\Gamma\left(\frac{3}{p}\right)^2}{4\Gamma\left(\frac{1}{p}\right)^2}(6-3-4)=-\frac{\Gamma\left(\frac{3}{p}\right)^2}{4\Gamma\left(\frac{1}{p}\right)^2}.
\end{eqnarray*}
As in Proposition 4.3, we bound the terms in Lemma 4.7 in the same manner and conclude as before that for some absolute constants,\medskip\medskip

$\displaystyle\qquad\qquad
\frac{np\E S^{p+3}\E \psi_{\theta_0}\overline{\psi}_{\theta_0}}{\E S^{p-1}}f(\xi_1,\xi_2)\leq-C_1n+C_2\sqrt{n}+C_3\sqrt{n}+C_4.
$\end{proof}\smallskip

\begin{proof}[Proof of Corollary 1.2]
It follows from Lemmas 4.8 and 4.10 and Theorem 1.2.
\end{proof}

\begin{rmk} By a closer study it is possible to state the results of this section letting $p$ grow with $n$. Using the estimate 
$\frac{c'\sqrt n}{\sqrt p}\le\E\psi_{\theta_0}\le\frac{C'\sqrt  n}{\sqrt p},\, p\geq n$, proven in \cite{AB3}, our method works for at least $p\le cn^2$. However, by viewing the situation at $p=\infty$ mentioned in the introduction, we believe  Corollary 1.3 should hold for $C$ independent of $p$.
\end{rmk}

 \end{document}